\documentclass[11pt]{article}

\usepackage{amsmath}
\usepackage{amssymb}
\usepackage{amsthm}
\usepackage{geometry}
\usepackage{makeidx}
\usepackage{enumerate}
\usepackage{stmaryrd}
\usepackage[english]{babel}
\usepackage[T1]{fontenc}
\usepackage{lmodern}
\usepackage[utf8]{inputenc}
\usepackage{pgf,tikz}
\usetikzlibrary{arrows}
\usepackage{relsize}

\theoremstyle {definition} \newtheorem {defi} {Definition} [section]
\theoremstyle {definition} 
\theoremstyle {plain}  \newtheorem {thm} [defi] {Theorem}
\theoremstyle {plain}  
\theoremstyle {plain}  \newtheorem {cor} [defi]{Corollary}
\theoremstyle {plain} \newtheorem {prop} [defi]{Proposition}
\theoremstyle {plain} \newtheorem {lem}[defi] {Lemma}
\theoremstyle {definition}  
\theoremstyle {definition}

\allowdisplaybreaks

\newcommand{\ddt}{\frac{\partial}{\partial t}}

\newcommand{\R}{\mathbb{R}}
\newcommand{\oo}{\omega}

\newcommand{\ddd}{\overrightarrow{\Delta}}
\newcommand{\tm}{\Lambda^1T^*M}
\newcommand{\tmk}{\Lambda^kT^*M}

\def\11{1\!\!1}

\def\f12{\frac{1}{2}}

\title{$L^p$-estimates for the heat semigroup on differential forms, and related problems}
\author{Jocelyn Magniez and El Maati Ouhabaz}

\begin{document}

\date{}
\maketitle

\noindent \begin{abstract} We consider a complete non-compact Riemannian manifold satisfying the volume doubling property and a Gaussian upper bound for its heat kernel (on functions). Let $\ddd_k $ be the Hodge-de Rham Laplacian  on differential $k$-forms with $k \ge 1$.  By the Bochner decomposition formula $\ddd_k = \nabla^* \nabla + R_k$. Under the assumption that the negative part $R_k^-$ is in an enlarged Kato class, we prove that for all $p \in [1, \infty]$, 
$\| e^{-t\ddd_k}\|_{p-p} \le C ( t \log t)^{\frac{D}{4}(1- \frac{2}{p})}$ (for large $t$). This estimate can be improved if $R_k^-$ is strongly sub-critical. In general, $(e^{-t\ddd_k})_{t>0}$ is not uniformly bounded on $L^p$ for any $p \not= 2$. 
We also prove the gradient estimate  $\|\nabla e^{-t\Delta}\|_{p-p} \le C t^{-\frac{1}{p}}$, where $\Delta$ is the Laplace-Beltrami operator (acting on functions).  Finally we discuss heat kernel bounds on forms and the Riesz transform on $L^p$ for $p > 2$. 
\end{abstract}

\section{Introduction and main results}\label{section51}

\noindent Let $(M,g)$ be a complete non-compact Riemannian manifold of dimension $N$, where $g$ denotes a Riemannian metric on $M$, 
Let $\rho$ and $\mu$ be the Riemannian distance and measure associated with $g$, respectively. We suppose that $M$ satisfies the volume doubling property, that is there exists constants $C,D>0$ such that 
\begin{equation}\tag{D}\label{D3}
v(x,\lambda r)\le C\lambda^{D}v(x,r),\;\forall x\in M,\forall r\ge 0, \forall\lambda\ge 1,
\end{equation}
\noindent where $v(x,r)=\mu(B(x,r))$ denotes the volume of the ball $B(x,r)$ of center $x$ and radius $r$.  It is a standard fact that this property is equivalent to the following one

\begin{equation*}
v(x,2 r)\le Cv(x,r),\;\forall x\in M,\forall r\ge 0.
\end{equation*}

\medskip

\noindent Let $\Delta$ be (the non-negative) Laplace-Beltrami operator and 
 $(e^{-t\Delta})_{t\ge 0}$  the associated heat semigroup.   It is a well known fact that  $(e^{-t\Delta})_{t\ge0}$ is a submarkovian semigroup. In particular, it acts as a contraction semigroup on $L^p(M)$ for all $p \in [1, \infty]$, i.e.,
 \begin{equation}\label{eq0}
 \|e^{-t\Delta}\|_{p-p} \le 1 \ \forall t\ge 0,\forall p\in [1,\infty].
 \end{equation}
The semigroup is strongly  continuous on $L^p(M)$ for $ 1 \le p < \infty$. 

 Now we introduce the Hodge-de Rham Laplacian $\ddd_k := d_k^*d_k + d_{k-1}d_{k-1}^*$ where $d_k$ denotes  the exterior derivative on differential $k$-forms and $d_k^*$ its formal adjoint. For $k=1$ we use the notation $\ddd$ instead of $\ddd_1$. The operator
 $\ddd_k$ is self-adjoint on $L^2(\tmk)$ where $\tmk$ denotes the space of smooth differential $k$-forms on $M$. Since $\ddd_k$ is non-negative, 
  $(e^{-t \ddd_k})_{t \ge0}$  is a contraction semigroup on $L^2(\tmk)$. One of the main questions we address in this paper concerns $L^p$-estimates of this semigroup. We formulate this as follows. 
 
  \vspace{.5cm}
  \underline{Question}: Suppose that $k \ge 1$. Does the semigroup  $(e^{-t \ddd_k})_{t \ge0}$ extend from $L^2(\tmk) \cap L^p(\tmk)$ to  $L^p(\tmk)$ and what is the behavior of its norm
  $\| e^{-t\ddd_k}\|_{p-p}$  on $L^p(\tmk)$ for {\it all or some} $p \not= 2$ ?
  
  \vspace{.5cm} 
The answer to this question is  intimately related to the geometry of the manifold $M$.  While it is not hard  to prove in a quite general setting  that $(e^{-t \ddd_k})_{t \ge0}$ extends to a strongly continuous semigroup on $L^p(\tmk)$, the precise estimate of its $L^p$-norm  turns to be  very complicated. We believe that  uniform boundedness ({\it w.r.t.}  $t$), at least when $k = 1$,  is related to deep questions such as boundedness on $L^p(M)$ of the Riesz transform (on functions) $d\Delta^{-1/2}$ on $L^p(M)$ (with value in $L^p(\tm)$).   We shall see that this is indeed the case under an additional assumption on $M$,  and  in general $(e^{-t \ddd})_{t \ge0}$ is not uniformly bounded on $L^p(\tm)$ for any $p \not= 2$. 
 
 According to  Bochner's  formula, $\ddd_k =\nabla^*\nabla+ R_k$ 
where  $\nabla$ denotes  the Levi-Civita connection and $R_k$ is a symmetric section of ${\rm End}(\tmk)$. For differential forms of order $1$, 
$\ddd =\nabla^*\nabla+ R$ where $R$ is identified with the Ricci curvature. The above formula allows  us to look at $\ddd_k$ as  a Schr\"odinger operator with the vector-valued "potential" $R_k$. Regarding $L^p$ estimates of $e^{-t\ddd}$ (or $e^{-t \ddd_k}$)  it is expected that the difficulty occurs in the setting of manifolds whose Ricci curvature has a nontrivial negative part. This is what happens with Schr\"odinger operators on functions  with potentials having a non-trivial negative part. When $k = 1$, the link of the previous question to the geometry of $M$ is promptly done via  the negative part of its Ricci curvature. The same observation can be made when $k \ge 2$. 
 
 Let now $p(t,x,y)$ be  the heat kernel on functions (the heat kernel of the Laplace-Beltrami operator $\Delta$). We assume throughout this paper that 
 $p(t,x,y)$  satisfies a Gaussian upper bound 

\begin{equation}\tag{G}\label{G3}
p(t,x,y)\le\frac{C}{v(x,\sqrt{t})}\exp\left(-c\frac{\rho^2(x,y)}{t}\right) \  \forall t> 0,
\end{equation}
where  $C$ and $c$ are positive constants. We denote by $H$ the operator $\nabla^*\nabla+R_k^+$ where $R_k^+$ denotes the positive part of $R_k$. It is a self-adjoint operator (defined by quadratic form method). The well known domination property  says that (in the pointwise sense)
\begin{equation}\label{eq1}
| e^{-tH} \omega | \le e^{-t\Delta} | \omega|
\end{equation}
for all $t > 0$ and $w \in L^2(\tmk)$. Therefore, it follows from \eqref{eq0} that  the semigroup $e^{-tH}$ acts as a contraction semigroup on $L^p(\tmk)$ for all $p \in [1, \infty)$. It also follows from (G) that the heat kernel $h_t(x,y)$ of $H$ satisfies a Gaussian upper bound. 
The operator $\ddd_k$ can be seen as the perturbation of $H$ by the negative "potential" $-R_k^-$, i.e., $\ddd = H - R_k^-$. In order to make this precise (using for example the method of quadratic forms to deal with this perturbation) we assume that $R_k^-$ is in the {\it enlarged Kato class} $\hat{K}$ which we introduce now.

\begin{defi}
We say that a function $f \in \hat{K}$ if there exists $\alpha>0$ such that
\begin{equation*}
\underset{x\in M}{\sup}\int_M \int_0^{\alpha}p_s(x,y) |f(y)| ds d\mu(y)<1. 
\end{equation*}
We say that $R_k^- \in \hat{K}$ if the function $x \mapsto | R_k^-(x) |_x$ belongs to $\hat{K}$. 
\end{defi}
\vspace{.2cm}

\noindent Note that $\hat{K}$ contains  the usual Kato class $K$, which is defined as the set of functions $f$ such that 
\begin{equation*}
\underset{\alpha\rightarrow 0}{\lim}\;\underset{x\in M}{\sup}\int_M \int_0^{\alpha}p_s(x,y) |f(y)| ds d\mu(y)=0. 
\end{equation*}

\noindent The Kato class $K$ plays an important role in the study of Schr\"odinger operators and their associated semigroups, see Simon \cite{Simon} and the references there. The class $\hat{K}$ appears in Voigt \cite{voigt} who studied properties of semigroups associated to Schr\"odinger operators with potential in $\hat{K}$ such as $L^p-L^q$ properties  for instance.

Note that the assumption that the  Ricci curvature is bounded from below means that  the negative part $R_-$ is bounded. In this case,
 $R_- \in \hat{K}$. Indeed,
$$\int_M \int_0^{\alpha}p_s(x,y) |R_-(y)|_y ds d\mu(y) \le \| R_- \|_\infty \int_0^\alpha \int_M p_s(x,y) d\mu(y) ds \le \alpha \| R_- \|_\infty.$$
The last inequality follows from \eqref{eq0} (with  $p= \infty$).  

Let $\overrightarrow{p_k}(t,x,y)$ denote the heat kernel of the operator $\ddd_k$.  The following is our first main result. 

\begin{thm}\label{theorem1}
\noindent Suppose that the manifold $M$ satisfies the volume doubling condition (\ref{D3}), the Gaussian upper bound  (\ref{G3}) and $R_k^- \in \hat{K}$. Then 
\begin{enumerate}[(i)]
\item the semigroup $(e^{-t \ddd_k})$ acts on $L^p(\tmk)$ for all  $p\in [1,\infty]$ and 
\begin{equation*}
\| e^{-t \ddd_k} \|_{p-p}\le C_p (t\log t)^{\left|\frac{1}{2}-\frac{1}{p}\right|\frac{D}{2}}, \ t > e.
\end{equation*}
\item For all $t\ge 1$ and $p\ge 2$
\begin{equation*}
\| \nabla  e^{-t\Delta}\|_{p-p}\le C_pt^{-\frac{1}{p}}.
\end{equation*}
\item There exists $C>0$ such that for all $t> 0$ and $x,y\in M$
\begin{equation*}
|\overrightarrow{p_k}(t,x,y)|\le C\frac{\left(1+ t+\frac{\rho^2(x,y)}{t}\right)^{\frac{D}{2}}}{v(x,\sqrt{t})^{\frac{1}{2}}v(y,\sqrt{t})^{\frac{1}{2}}}\exp\left(-\frac{\rho^2(x,y)}{4t}\right).
\end{equation*}

\item \noindent There exist $C, c >0$ such that for all $t \ge 1$ and $x,y\in M$
\begin{equation*}
|\overrightarrow{p_k}(t,x,y)|\le C\min\left(1,\frac{t^{\frac{D}{2}}}{v(x,\sqrt{t})}\right)\exp\left(-c\frac{\rho^2(x,y)}{t}\right).
\end{equation*}

\end{enumerate}
\end{thm}
Here and throughout this paper we use the notation $| \overrightarrow{p_k}(t,x,y) |$  for the norm from $\Lambda^kT^*_yM$ to $\Lambda^kT^*_xM$  of the linear map $\overrightarrow{p_k}(t,x,y)$ between these two spaces. Thus this norm depends on the fixed points $x$ and $y$. We make this   dependence implicit. 

Note that property $(iii)$ obviously gives 
\begin{equation*}
|\overrightarrow{p_k}(t,x,y)|\le C_\epsilon \frac{(1+ t)^{\frac{D}{2}}}{v(x,\sqrt{t})^{\frac{1}{2}}v(y,\sqrt{t})^{\frac{1}{2}}}\exp\left(-\frac{\rho^2(x,y)}{(4+\epsilon)t}\right), \ t > 0
\end{equation*}
for every $\epsilon > 0$ and 
In particular, 
\[
|\overrightarrow{p_k}(t,x,y)| \le  \frac{C}{v(x,\sqrt{t})} \exp\left(-c\frac{\rho^2(x,y)}{t}\right)
\]
for all $t \in (0, 1]$ and some  constant  $c > 0$. This is the reason why we formulate $(iv)$ and several  other estimates in this paper for $ t > 1$ only. 

Similar results  as in this theorem  (when $k = 1$) have been proved by Coulhon and Zhang \cite{coulz} under additional  assumptions. More precisely, it is assumed in \cite{coulz} a  "non collapsing" property  on the volume $v(x,r)$, the Ricci curvature is bounded from below and that the negative part $V(x)$ of the lowest eigenvalue of the Ricci tensor is strongly 
sub-critical (see below).   The latter  is rather a strong assumption. We do not make any of these assumptions and our condition $R_k^- \in \hat{K}$ is fairly general.

Assertion $(i)$ gives a partial answer to the question addressed above.  We do not know whether the behavior $(t\log t)^{\left|\frac{1}{2}-\frac{1}{p}\right|\frac{D}{2}}$ is optimal in general.  If  $M = \R^2 \sharp \R^2$ (connected sum of two copies of $\R^2$) then  $(e^{-t\ddd})_{t\ge 0}$ is not uniformly bounded on $L^p(\tm)$ for any $p \not=2$. \\
The gradient estimate in assertion $(ii)$ is formulated for $p > 2$ since for $p \in (1,2]$ the Riesz transform $d \Delta^{-1/2}$ is bounded on $L^p$ (cf. Coulhon and Duong \cite{could}). Therefore 
$\| d e^{-t\Delta}\|_{p-p}\le C_pt^{-\frac{1}{2}}$. The case $p > 2$ is complicated and this latter estimate is actually equivalent to the boundedness of the Riesz transform (at least under some additional assumptions on $M$, see Auscher et al. \cite{acdh}). Hence in the general setting of our paper, we cannot hope for gradient estimates of the semigroup in terms of  $t^{-\frac{1}{2}}$ (up to a constant) for all $p > 2$. \\
If in addition to the  assumptions made in \cite{coulz} (mentioned above) it is proved in the same paper that if $V \in L^p(M)$ then 
\begin{equation} \label{eq2}
|\overrightarrow{p_1}(t,x,y)|\le C\min\left(1,\frac{t^{\alpha_p}}{v(x,\sqrt{t})}\right)\exp\left(-c\frac{\rho^2(x,y)}{t}\right)
\end{equation}
for some constant $\alpha_p$ depending on $p$ and the sub-critical constant of $V$.  In particular,  $\alpha_p \to \infty $ (as $p \to \infty$) and hence  our estimate is more precise than \eqref{eq2}.   In addition, we do not make any sub-criticality assumption and no summability condition on $V$. 
We point out that recently, Coulhon, Devyver and Sikora \cite{CDS} were able to prove the full Gaussian bound
\begin{equation} \label{eq3}
|\overrightarrow{p_k}(t,x,y)| \le  \frac{C}{v(x,\sqrt{t})} \exp\left(-c\frac{\rho^2(x,y)}{t}\right),\  t > 0
\end{equation}
under additional assumptions on the Ricci curvature among which $R_k^-$ is strongly sub-critical and  {\it small at infinity} in a precise sense (see the next section). The estimate \eqref{eq3}  was proved previously by Devyver \cite{D} under the additional assumption that $M$ satisfies a global Sobolev inequality. 

\begin{cor}\label{cor31}
\noindent Suppose the assumptions of the previous theorem are satisfied. Then for every $\varepsilon > 0$, the local Riesz transform $d(\Delta+\varepsilon)^{-\frac{1}{2}}$ is bounded on $L^p(M)$  (into $L^p(\tm)$)  for all $p\in(1,\infty)$.
\end{cor}

As mentioned above, if $p \in (1,2]$ then the Riesz transform $d\Delta^{-1/2}$ is bounded on $L^p(M)$, see \cite{could}. The novelty here is the case $p > 2$ although we do not prove boundedness of the Riesz transform (we only treat the local one). A similar result for local Riesz transforms was proved by Bakry \cite{B} for Riemannian manifolds with Ricci curvature bounded from below (i.e., $R_- \in L^\infty$). We already mentioned that $\hat{K}$ is larger than 
$L^\infty$. \\ 
In order to prove the corollary we use follow  the ideas in \cite{coulhonduong}. By Theorem \ref{theorem1}, the heat kernel of $\varepsilon + \ddd$ has a full Gaussian bound (as in the RHS of \eqref{eq3}). This and the doubling condition imply  that $d^* (\varepsilon + \ddd)^{-1/2}$ is bounded from $L^q(\tm)$ to  $L^q(M)$ for $q \in (1, 2)$.  
Therefore by duality, $(\varepsilon + \ddd)^{-1/2} d$ is bounded from $L^p(M)$ to $L^p(\tm)$ for all $p \in (2, \infty)$. The commutation property
$\ddd d = d \Delta$ gives the desired result. We refer to  \cite{coulhonduong} for additional details and to \cite{acdh},  \cite{C2}, \cite{carron}, \cite{CMO}, \cite{D}, \cite{M} and the references therein for further results on the Riesz transform on $L^p$ for $p > 2$.


The next result is an improvement of  Theorem \ref{theorem1} under an additional assumption on $R_k^-$. 
We say that $R_k^-$ is $\epsilon$-sub-critical (or $\epsilon$-$H$-sub-critical or strongly sub-critical) for some $\epsilon \in [0,1)$ if 
\begin{equation}\label{eq4}
0\le (R_k^-\oo,\oo)\le\epsilon\,(H\oo,\oo) \  \forall \oo\in\mathcal{C}_0^{\infty}(\tm),
\end{equation}
where $H:=\nabla^*\nabla+R_k^+$ and $(., .)$ denotes the scalar product of  $L^2(\tm)$. For a scalar potential $V$, the definition of sub-criticality is the same, one  replaces  $H$ by $\Delta$, the scalar product is then taken  in $L^2(M)$. For further information on sub-criticality in the Euclidean space, see \cite{daviessimon}. 

\begin{thm}\label{theorem2}
\noindent Suppose that the manifold $M$ satisfies the volume doubling condition (\ref{D3}) with some $D>2$ and  the Gaussian upper bound  (\ref{G3}). Suppose in addition that $R_k^- \in \hat{K}$ and $R_k^-$ is  $\epsilon$-sub-critical for some $\epsilon \in [0, 1)$. Let $p_0:=\frac{2D}{(D-2)(1-\sqrt{1-\epsilon})}$ and fix $\tilde{p_0} < p_0$.   Denote by $p_0'$ the conjugate number of $p_0$. Then  
\begin{enumerate}[(i)]

\item We have
\begin{eqnarray*}
\|e^{-t\overrightarrow{\Delta}_k}\|_{p-p} &\le& C_p\text{ if }p\in (p_0',p_0), \ t \ge 0\\
\|e^{-t\overrightarrow{\Delta}_k}\|_{p-p} &\le& C_p(t\log t )^{\left|\frac{1}{\tilde{p_0}}-\frac{1}{p}\right|\frac{D}{2}}\text{ if } p\in [p_0,\infty], \ t > e\\
\|e^{-t\overrightarrow{\Delta}_k}\|_{p-p} &\le& C_p(t\log t )^{\left|\frac{1}{\tilde{p_0}} +\frac{1}{p}-1\right|\frac{D}{2}}\text{ if }p\in [1,p_0'], \ t > e.
\end{eqnarray*}

\item For all $t\ge 1$ 
\begin{equation*}
\|\nabla e^{-t\Delta}\|_{p-p}\le \frac{C_p}{\sqrt{t}}\text{ if }p\in (1,p_0)
\end{equation*}

\noindent and

\begin{equation*}
\|\nabla e^{-t\Delta}\|_{p-p}\le C_p \inf \left( t^{(\frac{1}{\tilde{p}_0}-\frac{1}{p})D-\frac{1}{2}}, t^{-\frac{\tilde{p_0}}{2p}} \right) \text{ if }p\ge p_0.
\end{equation*}

\item There exists $c,C>0$ such that for all $t>0$ and $x,y\in M$
\begin{equation*}
|\overrightarrow{p_k}(t,x,y)|\le \frac{C(1+ t)^{\frac{D}{\tilde{p}_0}}}{v(x,\sqrt{t})}\exp\left(-c\frac{\rho^2(x,y)}{t}\right).
\end{equation*}
\item For all $t\ge 1$ and $x,y\in M$
\begin{equation*}
|\overrightarrow{p_k}(t,x,y)|\le C\min\left(1,\frac{t^{\frac{D}{\tilde{p}_0}}}{v(x,\sqrt{t})}\right)\exp\left(-c\frac{\rho^2(x,y)}{t}\right).
\end{equation*}

\end{enumerate}
\end{thm}

The bounds given in this result are  better than those in Theorem \ref{theorem1}. For example, taking  $p = 1$ in the first assertion we obtain the following  $L^1-L^1$ estimate (or $L^\infty-L^\infty$ by duality) 
$$  \|e^{-t\overrightarrow{\Delta}_k}\|_{1-1}  \le C_\delta (t\log t )^{ \frac{D-2}{4}(1-\sqrt{1-\epsilon}) +\delta }$$
for all $\delta > 0$ and all large $t$. 
In the case where $D \le 2$ it is proved in \cite{M} that  the semigroup $(e^{-t\ddd})_{t\ge0}$ is uniformly bounded on $L^p(\tm)$ for all $p \in (1, \infty)$. The proof given there works for $\ddd_k$ and gives the same result for  $(e^{-t\ddd_k})_{t\ge0}$  for  any $k \ge 2$. We note however that this uses  the assumption that $R_k^-$ is strongly sub-critical. We already mentioned   above that if $M = \R^2 \sharp \R^2$ then $(e^{-t\ddd})_{t\ge0}$ is not uniformly bounded on $L^p(\tm)$ for any $p \not=2$.

The last result we mention in this introduction is the following.
\begin{prop}\label{proposition1}
Suppose that the assumptions of Theorem \ref{theorem2} are satisfied. If $(e^{-t\ddd})_{t\ge0}$ is uniformly bounded on $L^p(\tm)$ for some fixed $p$. Then the Riesz transform $d\Delta^{-1/2}$ is bounded on $L^r(M)$ for all $r \in (1, \max(p,p'))$.
\end{prop}

The paper is organized as follows. In Section \ref{section52}, we prove some useful preliminary results on the enlarged Kato class $\hat{K}$. Theorem \ref{theorem1} will be proved in Section \ref{section53}. Finally, Section \ref{section54} is devoted to the proofs of Theorem \ref{theorem2} and Proposition \ref{proposition1}. \\

Throughout this paper we use $C$ and $c$ for all inessential constants, their value may change from line to line. For a bounded linear operator $T: L^p \to L^q$  we use $\| T \|_{p-q}$ to denote the corresponding norm. For simplicity, we use the same notation $(., .)$ for the  scalar  product in $L^2(M)$ as well as in $L^2(\tmk)$ (the difference will be clear from the context). We also use the same notation for the duality $L^p-L^{p'}$.

\section{Preliminaries}\label{section52}

\noindent We first recall the following very well-known lemma.

\begin{lem}\label{integral}
Let $x\in M$. There exists a constant $C>0$ independent of $x$ such that for all $t>0$
\begin{equation*}
\int_M \exp\left(-c\frac{\rho^2(x,y)}{t}\right)d\mu(y)\le Cv(x,\sqrt{t}).
\end{equation*}
\end{lem}

\begin{proof}
\noindent We have
\begin{align*}
\int_M \exp\left(-c\frac{\rho^2(x,y)}{t}\right)d\mu(y)
&\le \sum_{j=0}^\infty\int_{j\sqrt{t}\le\rho(x,y)\le (j+1)\sqrt{t}}\exp(-cj^2)d\mu(y)\\
&\le \sum_{j=0}^\infty \exp(-cj^2)v(x,(j+1)\sqrt{t}).
\end{align*}
We use the  doubling property (\ref{D3}) and obtain the lemma.  
\end{proof}

Let $R_k$ be  an $L^\infty_{loc}$ section of the vector bundle ${\rm End}(\tmk)$. For each $x \in M$, the symmetric endomorphism 
$R_k(x)$ can be decomposed into a positive and negative parts $R_k(x)^+$ and $R_k(x)^-$ (i.e. $R_k(x) = R_k(x)^+ - R_k(x)^-$). A way to do this is to define $R_k(x)^+ w = R_k(x)w$ if $w$ is an eigenvector corresponding to a positive eigenvalue and $0$ if not. The negative part is
$R_k(x)^- = (-R_k(x))^+$. 

\begin{lem}\label{lemma2}
Suppose that $R_k^- \in \hat{K}$ and set $W(x) := | R_k^-(x)|_x$.  Then there exists  constants $\kappa \in [0,1)$ and $C > 0$ such that
$$\int_M W |f|^2 d\mu \le \kappa \| df \|_2^2 + C \|f \|_2^2$$
for all $f \in W^{1,2}(M)$.
\end{lem}
\begin{proof}
The arguments are the same as in \cite{voigt}, Proposition 4.7. We prove first  that there exists a sufficiently large $\lambda > 0$ such that 
$$\| (\Delta+\lambda)^{-1}W\|_{\infty-\infty}<1.$$
Indeed, 
\begin{align*}
\|(\Delta+\lambda)^{-1} W \|_{\infty-\infty}
&=\|\int_0^{\infty}e^{-\lambda s}e^{-s\Delta} W ds\|_{\infty-\infty}\\
&\le \sum_{n=0}^{\infty}\|\int_{n\alpha}^{(n+1)\alpha}e^{-\lambda s}e^{-s\Delta} W ds\|_{\infty-\infty}\\
&\le \sum_{n=0}^{\infty}\|\int_0^{\alpha}e^{-\lambda(t+n\alpha)}e^{-(t+n\alpha)\Delta} W dt\|_{\infty-\infty} \\
&\le \sum_{n=0}^{\infty}e^{-\lambda n\alpha}\|e^{-n\alpha\Delta}\|_{\infty-\infty}\|\int_0^{\alpha}e^{-\lambda t}e^{-t\Delta}W dt\|_{\infty-\infty}\\
&\le \frac{1}{1-e^{-\lambda \alpha}}\|\int_0^{\alpha}e^{-t\Delta} W dt\|_{\infty-\infty}.
\end{align*}
\noindent Since we assume $ W\in \hat{K}$, there exists $\alpha > 0$ such that  $\|\int_0^{\alpha}e^{-t\Delta} W dt\|_{\infty-\infty}<1$. Taking  $\lambda$ sufficiently large we obtain  $\|(\Delta+\lambda)^{-1} W \|_{\infty-\infty}<1$.

Next, it follows by duality that $\| W (\Delta+\lambda)^{-1}\|_{1-1}<1$. Now we apply  Stein's interpolation theorem to $F(z) :=W^z(\Delta+\lambda)^{-1}W^{1-z}$ and  obtain
$$\|W^{\frac{1}{2}}(\Delta+\lambda)^{-\frac{1}{2}}\|_{2-2}^2=\|W^{\frac{1}{2}}(\Delta+\lambda)^{-1}W^{\frac{1}{2}}\|_{2-2}<1.$$
This gives the desired assertion. 
\end{proof}

The operator $\ddd_k$ is defined via the quadratic form
\begin{equation}\label{eq21}
\overrightarrow{\mathfrak{a}}(\omega, \omega) = (d_{k-1}^* \oo, d_{k-1}^*\oo) + (d_k \oo, d_k \oo). 
\end{equation}
On the other hand, it is known by the Bochner formula that
\begin{equation}\label{bochner-k}
\ddd_k = \nabla^* \nabla + R_k
\end{equation}
where $R_k$ can be expressed in terms of the Riemann curvature (see e.g., \cite{Berard}, Section E). Again, if $k = 1$ then $R_1$ is identified with the Ricci curvature $R$ of $M$. 

The previous lemma allows us to look at $\ddd_k$ as  the perturbation of $H := \nabla^* \nabla + R_K^+$ by $- R_k^-$ in the quadratic form sense. That is, $\ddd_k$ is the operator associated with the quadratic form
\begin{equation}\label{eq211}
\overrightarrow{\mathfrak{a}}(\omega, \omega) = (H\omega, \omega) - (R_k^- \omega, \omega).
\end{equation}
Indeed, by the domination property (see \cite{HSU}, \cite{Berard} or \cite{Ou98}) we have
\begin{equation}\label{eq22}
\int_M | d | \oo | |^2  \le \int_M |\nabla \oo |^2.
\end{equation}
Thus by Lemma \ref{lemma2},
\begin{align*}
|(R_k^- \oo, \oo)| &\le  \kappa \| d |\oo|  \|_2^2 + C \| \oo \|_2^2\\
&\le \kappa \int_M |\nabla \oo |^2 + C \| \oo \|_2^2\\
&\le \kappa (H \oo, \oo) + C \| \oo \|_2^2.
\end{align*}
We can now apply  the well known KLMN theorem (see e.g. \cite{ouhabaz}, Chapter 1).

Following \cite{CDS},  a non-negative  function $W \in K^\infty$ (it is the called {\it small at infinity}) if there exists $A>0$ such that
\begin{equation}
\underset{x\in M}{\sup}\int_{M\setminus B(o,A)}g(x,y) W(y) d\mu(y)<1, 
\end{equation}
where $o\in M$ is a fixed point and $g(x,y)$ is the Green function of the Laplace-Beltrami operator (we assume that $g(x,y)$ exists). 
As mentioned in the introduction, one of the main assumptions  in \cite{CDS} in order to prove a Gaussian upper bound for the heat kernel 
of $\ddd_k$ is 
$R_k^- \in K^\infty$ (i.e., $x \mapsto | R_k^-(x)|_x \in K^\infty$). We point out that this class is smaller than the enlarged Kato class $\hat{K}$.

\begin{prop}\label{proposition2}
We have $L^\infty_{loc} (M) \cap K^{\infty}\subset\hat{K}$. In particular, if $R_k^- \in K^\infty$ then $R_k^- \in \hat{K}$. 
\end{prop}

\begin{proof}
Let $W \in L^\infty_{loc}(M) \cap K^\infty$.  For $A>0$
\begin{align*}
\int_M \int_0^{t}p_s(x,y) W(y) ds d\mu(y)
&\le \int_{B(o,A)} \int_0^{t}p_s(x,y) W(y) ds  d\mu(y)\\
& +\int_{M\setminus B(o,A)} \int_0^{\infty}p_s(x,y) W(y) ds d\mu(y)\\
&\le  \| W \|_{L^\infty(B(o,A))} \int_0^t \int_M p_s(x,y) d\mu(y)  ds \\
& +\int_{M\setminus B(o,A)} g(x,y) W(y) d\mu(y)\\
& \le t \| W \|_{L^\infty(B(o,A))} +\int_{M\setminus B(o,A)} g(x,y) W(y) d\mu(y).
\end{align*}
Taking $t$ small enough and using the definition of $W \in K^\infty$ we see that the last term is less than $1$. This means that $W \in \hat{K}$.
\end{proof}


\section{Proof of Theorem \ref{theorem1}}\label{section53}

\noindent In this section we prove the four statements of Theorem \ref{theorem1}. We divide the proof into several  steps.\\
In order to avoid repetition, we assume throughout this section that $M$ satisfies the doubling property \eqref{D3} and the Gaussian upper bound \eqref{G3}. We also assume that $R_k^- \in \hat{K}$.
We denote by $V(x)$  the lowest eigenvalue of the Riemann curvature (the Ricci curvature if $k=1$)  $R_k(x)$ and $V_-$ is the negative part of the function $V(x)$. 

\begin{lem}\label{lemma31} There exist positive constants $M$ and $\nu$ such that  
$$ \| e^{-t(\Delta -V_-)} \|_{1-1} \le M e^{\nu t}$$
\end{lem}
\begin{proof} The arguments of the proof are mainly borrowed from \cite{voigt}).\\
First, since $R_k^- \in \hat{K}$ then by Lemma \ref{lemma2} $V_-$ satisfies 
$$\int_M V_- |f|^2 d\mu \le \kappa \| df \|_2^2 + C \|f \|_2^2$$
for all $f \in W^{1,2}(M)$ and for some constant $\kappa < 1$. As explained in the previous section, by the KLMN theorem, the operator $\Delta -V_-$ is self-adjoint (with an appropriate domain) and hence the semigroup $ e^{-t(\Delta -V_-)} $ acts on $L^2(M)$. 
On the other hand, since $V_- \in \hat{K}$ there exist $\alpha > 0$ and $\kappa' < 1$ such that for $V_n := \inf(V_-, n) $ 
$$ \int_0^\alpha \| V_n e^{-s\Delta} f \|_1 ds \le \kappa' \| f \|_1.$$
The constants $\alpha$ and $\kappa'$ are independent of $n$. Therefore, by Miyadera perturbation theorem (cf. \cite{voigt}, Theorem 1.1) $e^{-t(\Delta - V_n)}$ is a strongly continuous semigroup on $L^1(M)$ and 
\begin{equation}\label{eq31}
\| e^{-t(\Delta - V_n)} f \|_1 \le M e^{\nu t} \| f \|_1
\end{equation}
with constants $M$ and $\nu$ independent of $n$. On the other hand, by classical monotone convergence theorems (for quadratic forms),  $e^{-t(\Delta - V_n)}$ converges strongly in $L^2(M)$ to $e^{-t(\Delta - V_-)}$ as $n \to \infty$.  An application of Fatou's lemma in \eqref{eq31} with $f \in L^1 \cap L^2$ gives \eqref{eq31} for 
$e^{-t(\Delta - V_-)}$ and $f \in \in L^1 \cap L^2$. Finally, we  argue by density to extend the estimate to all $f \in L^1$.   \end{proof}
\begin{lem}\label{lemma32}
The heat kernel $\overrightarrow{p_k}(t,x,y)$ of $\ddd_k$ has the following upper bound
\begin{equation*}
|\overrightarrow{p_k}(t,x,y)|\le \frac{C}{v(x,\sqrt{t})}\exp\left(-c\frac{\rho^2(x,y)}{t}\right) e^{\nu t},
\end{equation*}
where $C, c$ and $\nu$ are positive constants.
\end{lem}
\begin{proof}
By the well known domination (cf. \cite{HSU}, \cite{Berard} or \cite{Ou98}) we have the pointwise estimate
\begin{equation}\label{eq32}
| e^{-t \ddd_k} \oo(x) |_x \le e^{-t(\Delta - V_-)} |\oo(x)|_x.
\end{equation}
Thus, $| \overrightarrow{p_k}(t,x,y)|$ is  bounded by the heat kernel $p^{V_-}(t,x,y)$ of the Schr\"odinger operator $\Delta - V_-$.
It is then enough to prove the above upper bound for $p^{V_-}(t,x,y)$. 

Since the heat kernel $p(t,x,y)$ of $\Delta$ has a Gaussian  upper bound \eqref{G3} it follows that the following Gagliardo-Nirenberg type inequality holds
\begin{equation}\label{eq33}
\|v(.,\sqrt{t})^{\frac{1}{2}-\frac{1}{q}}u\|_q\le C(\|u\|_2+\sqrt{t}\|\nabla u\|_2),
\end{equation} 
for all  $q \in (2, \infty)$  such that $\frac{q-2}{q}D<2$ (and $u\in W^{1,2}(M)$). See \cite{ao}, Proposition 2.1 and p. 1125 or \cite{bcs}, Theorem 1.2.1.
On the other hand, using Lemma \ref{lemma2} we see that for $\nu > 0$ large enough,
$$ \| \nabla u \|_2^2 \le C \| (\Delta - V_- + \nu)^{1/2} u \|_2^2.$$
Therefore, \eqref{eq33} implies 
\begin{equation}\label{eq34}
\|v(.,\sqrt{t})^{\frac{1}{2}-\frac{1}{q}}u\|_q\le C(\|u\|_2+\sqrt{t}\| (\Delta - V_- + \nu)^{1/2} u \|_2).
\end{equation} 
This together with the fact that the semigroup $(e^{-t(\Delta - V_- + \nu)})_{t\ge0}$ is uniformly bounded on $L^1(M)$ (cf. Lemma \ref{lemma31}) allows us to apply \cite{bcs}, Theorem 1.2.1  and conclude that $p^{V_-}(t,x,y) e^{-\nu t}$ satisfies a full Gaussian upper bound. One missing thing before applying the result from \cite{bcs} is that the semigroup
$T_t := e^{-t(\Delta -V_- + \nu)}$ has to satisfy  $L^2-L^2$ Davies-Gaffney estimates
$$\|T_t u\|_{L^2(F)}\le e^{-\frac{\rho^2(E,F)}{ct}}\|u\|_2,$$
where $E, F$ are two closed subsets of $M$ and $u$ has support in $E$. This is indeed true and it is a standard fact for Schr\"odinger operators. 
\end{proof}


In order to improve the Gaussian upper bound given here and obtain the bound $(iii)$ of Theorem \ref{theorem1} we take advantage of the fact that $(e^{-t\ddd_k})_{t\ge0}$ is uniformly bounded on $L^2(\tmk)$ (it is even a contraction semigroup since $\ddd_k$ is non-negative). This strategy was used in \cite{Ouhabaz2} and we follow the same arguments which we adapt to the setting of differential forms. 
\begin{prop}\label{semigaussian1}
We have 
\begin{equation}\label{upperbound2}
|\overrightarrow{p_k}(t,x,y)|\le\frac{C\left(1+ t+\frac{\rho^2(x,y)}{t}\right)^{\frac{D}{2}}}{v(x,\sqrt{t})^{\frac{1}{2}}v(y,\sqrt{t})^{\frac{1}{2}}}\exp\left(-\frac{\rho^2(x,y)}{4t}\right),
\end{equation}
for  some positive constant $C$. 
\end{prop}

\begin{proof}
We apply  Davies's perturbation method. Let $\lambda\in\mathbb{R}$ and $\phi\in\mathcal{C}_0^{\infty}(M)$ such that $|\nabla\phi|\le 1$ on $M$. We consider the semigroup $\overrightarrow{T}_{t,\lambda}:=e^{-\lambda\phi}e^{-t\overrightarrow{\Delta_k}}e^{\lambda\phi}$ and its integral kernel

$$\overrightarrow{k}_{\lambda}(t,x,y)=e^{-\lambda(\phi(x)-\phi(y))}\overrightarrow{p_k}(t,x,y).$$  

\noindent \underline{Step 1}. As a  consequence of $|\nabla\phi|\le 1$ and Lemma \ref{lemma32} we have
\begin{align*}
|\overrightarrow{k}_{\lambda}(t, x,y)|
&\le \frac{Ce^{\nu t}e^{|\lambda||\phi(x)-\phi(y)|}}{v(x,\sqrt{t})}\exp\left(-c\frac{\rho^2(x,y)}{t}\right)\\
&\le \frac{Ce^{\nu t}}{v(x,\sqrt{t})}\exp\left(-c\frac{\rho^2(x,y)}{t}+|\lambda|\rho(x,y)\right)\\
&\le \frac{Ce^{\nu t}}{v(x,\sqrt{t})}\exp\left(\frac{1}{2c}\lambda^2 t\right)\exp\left(-\frac{c}{2}\frac{\rho^2(x,y)}{t}\right).
\end{align*}

\noindent \underline{Step 2}. We prove that there exists a constant $C$ independent of $\lambda$ and $\phi$ such that for all $t>0$, $x,y\in M$ and $\lambda\in\mathbb{R}$
\begin{equation}\label{equa33}
\int_M |\overrightarrow{k}_{\lambda}(\frac{t}{2}, x,y)|^2d\mu(y)\le \frac{Ce^{\lambda^2 t}}{v\left(x,\sqrt{\min(\frac{t}{2},\frac{1}{\beta})}\right)},
\end{equation}

\noindent where $\beta:=(\frac{1}{2c}-1)\lambda^2+\nu$. One can obviously take the constant $c$ small enough in the estimate of Step 1 
so that  $\frac{1}{2c}>1$. Note that if $\beta=0$, then $\lambda=\nu=0$ and (\ref{equa33}) follows from the estimate in Step 1 and  Lemma \ref{integral}.  Hence we may assume in the sequel that  $\beta>0$. \\

\noindent We fix $t>0$. According to Step 1, if $t\le\frac{1}{\beta}$
\begin{align*}
\int_M |\overrightarrow{k}_{\lambda}(t,x,y)|^2d\mu(y)
&\le \frac{Ce^{(\frac{1}{c}\lambda^2+2\nu)t}}{v(x,\sqrt{t})^2}\int_M e^{-c\frac{\rho^2(x,y)}{t}}d\mu(y) \\
&\le \frac{Ce^{2\beta t}}{v(x,\sqrt{t})}e^{2t\lambda^2}\le\frac{Ce^2}{v(x,\sqrt{t})}e^{2t\lambda^2}.
\end{align*}

\noindent Thus (\ref{equa33}) holds  for all $t\le\frac{1}{\beta}$. \\
Next, we suppose $t>\frac{1}{\beta}$. The semigroup property implies 
\begin{equation}\label{equa4}
\int_M |\overrightarrow{k}_{\lambda}(t,x,y)|^2d\mu(y)\le \left\|\overrightarrow{T}_{t-\frac{1}{\beta},\lambda}\,\overrightarrow{k}_{\lambda}(\frac{1}{\beta}, x,.)\right\|_2^2\le e^{2(t-\frac{1}{\beta})\lambda^2}\left\|\overrightarrow{k}_{\lambda}(\frac{1}{\beta}, x,.)\right\|_2^2.
\end{equation}

\noindent The last inequality uses 

$$\|\overrightarrow{T}_{t,\lambda}\|_{2,2}\le e^{\lambda^2 t},\forall t\ge 0,$$

\noindent which follows from that fact that the operator  $\overrightarrow{A}_{\lambda}+\lambda^2$ is positive,  
where $-\overrightarrow{A}_{\lambda}$ denotes the generator of the semigroup $(\overrightarrow{T}_{t,\lambda})_{t\ge 0}$. For more details see the proof of \cite{M} Proposition 3.7 in the case $k = 1$, the arguments works for general $k \ge 1$.  Now we use the inequality

$$\left\|\overrightarrow{k}_{\lambda}(\frac{1}{\beta}, x,.)\right\|_2^2\le \frac{C}{v(x,\sqrt{\frac{1}{\beta}})}e^{\frac{2\lambda^2}{\beta}},$$

\noindent proved above (in the case $t\le \frac{1}{\beta}$) to obtain 

$$\left\|\overrightarrow{k}_{\lambda}(t, x,.)\right\|_2^2\le \frac{C}{v(x,\sqrt{\frac{1}{\beta}})}e^{2t\lambda^2}.$$

\noindent This proves (\ref{equa33}).\\

\noindent \underline{Step 3}. We prove that for all $t>0$ and $x,y\in M$
\begin{equation}\label{equa5}
|\overrightarrow{k}_{\lambda}(t,x,y)|\le \frac{Ce^{\lambda^2 t}}{v\left(x,\sqrt{\min(\frac{t}{2},\frac{1}{\beta})}\right)^{\frac{1}{2}}v\left(y,\sqrt{\min(\frac{t}{2},\frac{1}{\beta})}\right)^{\frac{1}{2}}}.
\end{equation}

\noindent First, changing $\lambda$ into $-\lambda$ in Step 2 gives by duality
\begin{equation}\label{equa6}
\int_M |\overrightarrow{k}_{\lambda}(\frac{t}{2},x,y)|^2d\mu(x)\le \frac{Ce^{\lambda^2 t}}{v\left(y,\sqrt{\min(\frac{t}{2},\frac{1}{\beta})}\right)}.
\end{equation}

\noindent The semigroup property implies

$$|\overrightarrow{k}_{\lambda}(t,x,y)|\le\int_M |\overrightarrow{k}_{\lambda}(\frac{t}{2},x,z)||\overrightarrow{k}_{\lambda}(\frac{t}{2},z,y)|d\mu(z).$$

\noindent Using the Cauchy-Schwarz inequality, (\ref{equa33}) and (\ref{equa6}), we obtain
\begin{align*}
|\overrightarrow{k}_{\lambda}(t,x,y)|
&\le\left(\int_M |\overrightarrow{k}_{\lambda}(\frac{t}{2},x,z)|^2 d\mu(z)\right)^{\frac{1}{2}}\left(\int_M |\overrightarrow{k}_{\lambda}(\frac{t}{2},z,y)|^2 d\mu(z)\right)^{\frac{1}{2}}\\
&\le \frac{Ce^{\lambda^2 t}}{v\left(x,\sqrt{\min(\frac{t}{2},\frac{1}{\beta})}\right)^{\frac{1}{2}}v\left(y,\sqrt{\min(\frac{t}{2},\frac{1}{\beta})}\right)^{\frac{1}{2}}}.
\end{align*}

\noindent \underline{Step 4}. We claim  that for all $t>0$ and $x,y\in M$
\begin{equation}\label{equa37}
|\overrightarrow{p_k}(t,x,y)|\le\frac{C}{v(x,\sqrt{r})^{\frac{1}{2}}v(y,\sqrt{r})^{\frac{1}{2}}}\exp\left(-\frac{\rho^2(x,y)}{4t}\right),
\end{equation}

\noindent where 

$$r:= \min\left(\frac{t}{2},\left[\left(\frac{1}{2c}-1\right)\frac{\rho^2(x,y)}{4t^2}+\nu\right]^{-1}\right).$$

\noindent The estimate (\ref{equa5}) and the definition of $\overrightarrow{k}_{t,\lambda}(x,y)$ give

$$|\overrightarrow{p_k}(t,x,y)|\le\frac{Ce^{\lambda^2 t}}{v\left(x,\sqrt{\min(\frac{t}{2},\frac{1}{\beta})}\right)^{\frac{1}{2}}v\left(y,\sqrt{\min(\frac{t}{2},\frac{1}{\beta})}\right)^{\frac{1}{2}}}e^{\lambda(\phi(x)-\phi(y))}.$$

\noindent Choosing $\lambda=\frac{\phi(y)-\phi(x)}{2t}$, we obtain

$$|\overrightarrow{p_k}(t,x,y)|\le\frac{C}{v\left(x,\sqrt{\min(\frac{t}{2},\frac{1}{\beta})}\right)^{\frac{1}{2}}v\left(y,\sqrt{\min(\frac{t}{2},\frac{1}{\beta})}\right)^{\frac{1}{2}}}\exp\left(-\frac{|\phi(x)-\phi(y)|^2}{4t}\right),$$

\noindent with

$$\beta= \left(\frac{1}{2c}-1\right)\frac{|\phi(x)-\phi(y)|^2}{4t^2}+\nu.$$

\noindent Since $|\nabla\phi|\le 1$, we have $|\phi(x)-\phi(y)|\le\rho(x,y)$. We deduce that

$$|\overrightarrow{p_k}(t,x,y)|\le\frac{C}{v\left(x,\sqrt{r}\right)^{\frac{1}{2}}v\left(y,\sqrt{r}\right)^{\frac{1}{2}}}\exp\left(-\frac{|\phi(x)-\phi(y)|^2}{4t}\right),$$

\noindent with $r$ as above. We optimize over $\phi$ and obtain (\ref{equa37}).\\

\noindent \underline{Step 5}. We deduce (\ref{upperbound2}) using assumption (\ref{D3}). Indeed, noting that 

$$v(x,\sqrt{t})\le v(x,\sqrt{r})\left(\frac{t}{r}\right)^{\frac{D}{2}}$$

\noindent and

$$\frac{t}{r}\le 2+\left(\frac{1}{2c}-1\right)\frac{\rho^2(x,y)}{4t}+\nu t,$$

\noindent we obtain for all $t>0$ and $x,y\in M$

$$|\overrightarrow{p_k}(t,x,y)|\le\frac{C}{v(x,\sqrt{t})^{\frac{1}{2}}v(y,\sqrt{t})^{\frac{1}{2}}}\exp\left(-\frac{\rho^2(x,y)}{4t}\right)\left(1+\nu t+\frac{\rho^2(x,y)}{t}\right)^{\frac{D}{2}}.$$
\end{proof}

\noindent The following result was  proved in \cite{coulz} under some additional assumptions. Since the ideas of the proof are the same we just sketch them.

\begin{prop}\label{semigaussian2}
With the same assumptions as in the previous theorem, there exists $c,C>0$ such that for all $t\ge 1$ and $x,y\in M$
\begin{equation*}
|\overrightarrow{p_k}(t,x,y)|\le C\exp\left(-c\frac{\rho^2(x,y)}{t}\right).
\end{equation*}
\end{prop}

\begin{proof}
Following the proof of \cite{coulz} Theorem 4.1, Step 1, we find for $A>0$ sufficiently large and $t\ge 1$

$$\int_M |\overrightarrow{p_k}(t,x,y)|^2 e^{\frac{\rho^2(x,y)}{At}}d\mu(x)\le\int_M |\overrightarrow{p_k}(1,x,y)|^2e^{\frac{\rho^2(x,y)}{A}}d\mu(x).$$

\noindent The estimate (\ref{upperbound2}) gives 

$$|\overrightarrow{p_k}(1,x,y)|\le \frac{C}{v(x,1)}e^{-c\rho^2(x,y)}.$$

\noindent We deduce that for $A>0$ sufficiently large and all $t\ge 1$

$$\int_M |\overrightarrow{p_k}(t,x,y)|^2 e^{\frac{\rho^2(x,y)}{At}}d\mu(x)\le C.$$

\noindent Then the rest of the proof is the same as in \cite{coulz} Theorem 4.1, Step 1.
\end{proof}

\noindent We deduce from Propositions \ref{semigaussian1} and \ref{semigaussian2} the following estimate which is the estimate $(iv)$ of Theorem \ref{theorem1}

\begin{equation}
|\overrightarrow{p_k}(t,x,y)|\le C\min\left(1,\frac{t^{\frac{D}{2}}}{v(x,\sqrt{t})}\right)\exp\left(-c\frac{\rho^2(x,y)}{t}\right).
\end{equation}

\noindent Now we prove assertion $(ii)$ of Theorem \ref{theorem1}.  This result can be found in \cite{coulhonduong} where the authors suppose that $\overrightarrow{p_k}(t,x,y)$ satisfies a full Gaussian upper bound instead of (\ref{upperbound2}). We give the proof for the sake of completeness.

\begin{prop}\label{gradient}
For all $2\le p\le\infty$ and $t\ge 1$, 
$$|\nabla_x p(t,x,y)|\leqslant \frac{Ct^{\frac{D}{2}}}{\sqrt{t}\,v(x,\sqrt{t})}e^{-c\frac{\rho^2(x,y)}{t}},\;\forall x,y\in M,\forall t>1.$$
and 
$$\|\nabla e^{-t\Delta}\|_{p,p}\le C t^{-\frac{1}{p}}.$$
\end{prop}

\begin{proof}
The Gaussian upper bound \eqref{G3} and the doubling condition \eqref{D3} imply that there exist $\nu > 0$ small enough such that 
	
\begin{equation}\label{eqgra}
\int_M |\nabla_x p(t,x,y)|^2e^{\nu \frac{\rho^2(x,y)}{t}} d\mu(x)\leqslant \frac{C}{t\,v(y,\sqrt{t})},  \ y \in M,   t > 0
\end{equation}
This is proved in \cite{Gr2} (see also \cite{could}). In addition, by 
the Cauchy-Schwarz inequality we have for  $\gamma > 0$ small enough,
\begin{align*}
&\int_M |\nabla_x p(t,x,y)|e^{\gamma \frac{\rho^2(x,y)}{t}} d\mu(x)\\
&\leqslant \left(\int_M |\nabla_x p(t,x,y)|^2 e^{4\gamma\frac{\rho^2(x,y)}{t}}d\mu(x)\right)^{\frac{1}{2}} \left(\int_M e^{-2\gamma\frac{\rho^2(x,y)}{t}}d\mu(x)\right)^{\frac{1}{2}}\\
&\leqslant \left(\frac{C_{\gamma}}{t\,v(y,\sqrt{t})}\right)^{\frac{1}{2}}(C_{\gamma}'v(y,\sqrt{t}))^{\frac{1}{2}}\\
&=\frac{C_{\gamma}''}{\sqrt{t}},
\end{align*}
where we used \eqref{eqgra} and Lemma \ref{integral}. Hence
\begin{equation}\label{eqgra1}
\int_M |\nabla_x p(t,x,y)|e^{\gamma \frac{\rho^2(x,y)}{t}} d\mu(x)\leqslant \frac{C_{\gamma}}{\sqrt{t}}.
\end{equation}
Now we prove that there exist  $c, C>0$ such that
\begin{equation}\label{step4}
|\nabla_x p(t,x,y)|\leqslant \frac{Ct^{\frac{D}{2}}}{\sqrt{t}\,v(x,\sqrt{t})}e^{-c\frac{\rho^2(x,y)}{t}},\;\forall x,y\in M,\forall t>1.
\end{equation}

\noindent First notice that
$$\nabla_x p(t,x,y)=\int_M \nabla_x p(\frac{t}{2}, x,z)p(\frac{t}{2}, z,y)d\mu(z)=\nabla_x e^{-\frac{t}{2}\Delta}f_{y,t}(x),$$

\noindent with $f_{y,t}(z)=p(\frac{t}{2}, z,y)$. According to the commutation formula $\overrightarrow{\Delta}d=d\Delta$ (recall that 
$\ddd = \ddd_1$)  we have 
$$|\nabla_x e^{-\frac{t}{2}\Delta}f_{y,t}(x)|=|e^{-\frac{t}{2}\overrightarrow{\Delta}}df_{y,t}(x)|.$$

\noindent Therefore using the estimate (\ref{upperbound2}) (in the case $k=1$) and (\ref{D3}), we obtain for all $t>1$ and some constant $c > 0$
\begin{align*}
|\nabla_x e^{-\frac{t}{2}\Delta}f_{y,t}(x)|
&\leqslant \frac{Ct^{\frac{D}{2}}}{v(x,\sqrt{t})}\int_M e^{-c\frac{\rho^2(x,z)}{t}}|df_{y,t}(z)|d\mu(z)\\
& =\frac{Ct^{\frac{D}{2}}}{v(x,\sqrt{t})}\int_M e^{-c\frac{\rho^2(x,z)}{t}}|\nabla_z f_{y,t}(z)|d\mu(z).
\end{align*}

\noindent Hence
$$|\nabla_x p(t,x,y)|\leqslant \frac{Ct^{\frac{D}{2}}}{v(x,\sqrt{t})}\int_M e^{-c\frac{\rho^2(x,z)}{t}}|\nabla_z f_{y,t}(z)|d\mu(z).$$

\noindent It suffices then to prove that for some $\gamma<c$, there exist constants $c',C'>0$ such that
$$\int_M e^{-\gamma\frac{\rho^2(x,z)}{t}}|\nabla_z p(\frac{t}{2}, z,y)|d\mu(z)\leqslant \frac{C'}{\sqrt{t}}e^{-c'\frac{\rho^2(x,y)}{t}}.$$

\noindent Using
$$e^{-\gamma\frac{\rho^2(x,z)}{t}}\leqslant e^{-\gamma\frac{\rho^2(x,y)}{2t}}e^{\gamma\frac{\rho^2(z,y)}{t}},$$

\noindent we have
$$\int_M e^{-\gamma\frac{\rho^2(x,z)}{t}}|\nabla_z p(\frac{t}{2}, z,y)|d\mu(z)\leqslant e^{-\gamma\frac{\rho^2(x,y)}{2t}}\int_M e^{\gamma\frac{\rho^2(z,y)}{t}}|\nabla_z p(\frac{t}{2},z,y)|d\mu(z).$$

\noindent According to \eqref{eqgra1}, we know that there exists $\gamma>0$ small enough such that
$$\int_M e^{\gamma\frac{\rho^2(z,y)}{t}}|\nabla_z p(\frac{t}{2},z,y)|d\mu(z)\leqslant \frac{C'}{\sqrt{t}}.$$
This shows  \eqref{step4}. 
	
If follows from Lemma \ref{integral} and \eqref{step4}  that 
\begin{equation}\label{interpolation2}
\|\nabla e^{-t\Delta}\|_{\infty-\infty}\le\frac{Ct^{\frac{D}{2}}}{\sqrt{t}}. 
\end{equation}
We write for $t \ge 1$, $\nabla e^{-t\Delta} = \nabla e^{-\Delta} e^{-(t-1) \Delta}$. We apply 
\eqref{interpolation2} with $t=1$ and the standard estimate 
\begin{equation}\label{interpolation1}
\|\nabla e^{-t\Delta}\|_{2-2}\le\frac{C}{\sqrt{t}},
\end{equation}
to obtain 
$$\|\nabla e^{-t\Delta}\|_{p-p}\le C t^{-\frac{1}{p}}.$$
\end{proof}

 Finally using the estimate in assertion $(iii)$ we prove  assertion $(i)$ in Theorem \ref{theorem1} as in \cite{Ouhabaz2}, Th\'eor\`eme 7.  In order to avoid repetition, we do not give details here and we come back to this in the proof of Theorem \ref{theorem2}, assertion $(i)$. \\

 \noindent{\bf Example}. Let $M = \R^2 \sharp \R^2$. It is proved in \cite{CDS} Proposition 3.3,  that if the gluing is made in such a way that $M$ has genus zero then
 $\ker_{L^2}(\ddd) = \{0 \}$ and $\ker_{L^p}(\ddd) \not= \{0\}$ for all $p > 2$. We claim that $(e^{-t\ddd})_{t\ge0}$ is not uniformly bounded on $L^p(\tm)$ for all $p \not= 2$. \\
 Suppose for a contradiction that $(e^{-t\ddd})_{t\ge0}$ is uniformly bounded on $L^{p_0}(\tm)$ for some $p_0 \not=2$. Arguing by duality, we may assume that $p_0 > 2$. For $\oo = \ddd \eta  \in R_{L^2}(\ddd)$, the range of $\ddd$ as an operator on $L^2(\tm)$, we have 
 $$ \| e^{-t\ddd} \oo \|_2  = \| \ddd e^{-t \ddd} \eta \|_2  \le \frac{C}{t} \| \eta \|_2.$$
 In particular, $e^{-t \ddd} \oo \to 0$ in $L^2(\tm)$ as $t \to \infty$. Since $R_{L^2}(\ddd)$ is dense in $L^2(\tm)$ (thanks to the fact that $\ker_{L^2}(\ddd) = \{0 \}$), $e^{-t \ddd} \oo \to 0$ (as $t \to \infty$) for all $\oo \in L^2(\tm)$. Let $p \in (2, p_0)$ and $\oo \in L^2(\tm) \cap L^{p_0}(\tm)$. The classical interpolation inequality together with the assumption that $(e^{-t\ddd})_{t\ge0}$ is uniformly bounded on $L^{p_0}(\tm)$ imply that $e^{-t \ddd} \oo \to 0$ in $L^p(\tm)$ as $t \to \infty$. Uniform boundedness of $(e^{-t\ddd})_{t\ge0}$ on $L^p(\tm)$ implies that this holds for all $\oo \in L^p(\tm)$. This cannot be the case since $\ker_{L^p}(\ddd) \not= \{0\}$. 
 

\section{Proof of Theorem \ref{theorem2}}\label{section54}

Recall that $p_0:=\frac{2D}{(D-2)(1-\sqrt{1-\epsilon})}$. It is shown in \cite{M} that the semigroup $(e^{-t\ddd})_{t\ge0}$ is uniformly bounded on $L^p(\tm)$ for all $p \in (p_0', p_0)$. The  proof uses perturbation arguments and it works for the semigroup $(e^{-t\ddd_k})_{t\ge0}$ as well for any $k \ge 1$.  Hence the first estimate in assertion $(i)$ of Theorem \ref{theorem2} is known. In addition, it is also proved in \cite{M} that the Riesz transform $d \Delta^{-1/2}$ is bounded on $L^p$ for $p \in (2, p_0)$ and as we already mentioned before, the Riesz transform is bounded on $L^p$ for $p \in (1, 2]$ (cf. \cite{could}). This implies the first estimate in assertion $(ii)$. 
Now we prove the remaining estimates  of Theorem \ref{theorem2}. 

\begin{prop}\label{propouhabaz}
Assume that (\ref{D3}), (\ref{G3}) and $R_k^- \in \hat{K}$.  If the semigroup $(e^{-t \ddd_k})_{t\ge 0}$ is uniformly bounded on $L^p(\tmk)$ for some $p\in (1,2]$ then for  all $t>e$
$$\|e^{-t \ddd_k}\|_{\infty-\infty}\le C (t\log t )^{\frac{D}{2}\left(1-\frac{1}{p}\right)}.$$
\end{prop}

Since $(e^{-t\ddd_k})_{t\ge 0}$ is uniformly bounded on $L^p(\tmk)$ for $p \in (p_0', p_0)$ we obtain immediately assertion $(i)$ of Theorem \ref{theorem2} for $p \in [p_0, \infty]$. The case $p \in [1, p_0']$ is handled by duality.

\begin{proof}
The idea of proof is  similar to \cite{Ouhabaz2}, Th\'eor\`eme 7. Let $t> e$. Since the semigroup $(e^{-t\ddd_k})_{t\ge 0}$ is uniformly bounded on $L^p$, we have 
\begin{equation*}
\int_M |\overrightarrow{p_k}(t,x,y)|^p d\mu(y)\le \|e^{-(t-1)\overrightarrow{\Delta_k}}\|_{p-p}^p \;\| \overrightarrow{p_k}(1,x,.) \|_p^p\le C\int_M|\overrightarrow{p_k}(1,x,y)|^pd\mu(y).
\end{equation*}

\noindent Then using (\ref{upperbound2}) and Lemma \ref{integral} we obtain

$$\int_M|\overrightarrow{p_k}(1,x,y)|^pd\mu(y)\le \frac{C}{v(x,1)^p}\int_M \exp\left(-cp\rho^2(x,y)\right)\le \frac{C}{v(x,1)^{p-1}}.$$

\noindent Hence 
\begin{equation}\label{equa310}
\int_M |\overrightarrow{p_k}(t,x,y)|^p d\mu(y)\le\frac{C}{v(x,1)^{p-1}}.
\end{equation}

\noindent Let $\epsilon\in(0,1)$ and $p_{\epsilon}\in (0,1)$ such that $\frac{1-\epsilon}{p}+\frac{\epsilon}{p_{\epsilon}}=1$, that is $p_{\epsilon}=\frac{\epsilon p}{p-1+\epsilon}$. Using (\ref{upperbound2}) and  Lemma \ref{integral}, we have

$$\int_M |\overrightarrow{p_k}(t,x,y)|^{p_\epsilon} d\mu(y)\le\frac{C t^{\frac{Dp_{\epsilon}}{2}}}{v(x,\sqrt{t})^{p_{\epsilon}}}\int_M \exp\left(-cp_{\epsilon}\frac{\rho^2(x,y)}{t}\right)d\mu(y)\le \frac{C t^{\frac{Dp_{\epsilon}}{2}}}{v(x,\sqrt{t})^{p_{\epsilon}}}v(x,\sqrt{\frac{t}{p_{\epsilon}}}).$$

\noindent Using (\ref{D3}) we it follows that 
\begin{equation}\label{equa311}
\int_M |\overrightarrow{p_k}(t,x,y)|^{p_\epsilon} d\mu(y)\le C t^{\frac{Dp_{\epsilon}}{2}}v(x,\sqrt{t})^{1-p_{\epsilon}}p_{\epsilon}^{-\frac{D}{2}}.
\end{equation}

\noindent From (\ref{equa310}), (\ref{equa311}) and the  H\"older inequality, we deduce that
\begin{align*}
\int_M |\overrightarrow{p_k}(t,x,y)| d\mu(y)
&\le\left(\int_M |\overrightarrow{p_k}(t,x,y)|^p d\mu(y)\right)^{\frac{1-\epsilon}{p}}\left(\int_M |\overrightarrow{p_k}(t,x,y)|^{p_\epsilon} d\mu(y)\right)^{\frac{\epsilon}{p_{\epsilon}}}\\
&\le C \left(\frac{v(x,\sqrt{t})}{v(x,1)}\right)^{\frac{(p-1)(1-\epsilon)}{p}}t^{\frac{\epsilon D}{2}}p_{\epsilon}^{-\frac{D(p-1+\epsilon)}{2p}}\\
&\le Ct^{\frac{D(p-1)}{2p}}t^{\frac{\epsilon D}{2p}}p_{\epsilon}^{-\frac{D(p-1+\epsilon)}{2p}}.
\end{align*}
We have  from the definition of $p_{\epsilon}$ that  $p_{\epsilon}^{-(p-1+\epsilon)}\le C \epsilon^{1-p}$. Hence
\begin{equation}\label{equa312}
\int_M |\overrightarrow{p_k}(t,x,y)| d\mu(y)\le Ct^{\frac{D(p-1)}{2p}}\left[t^{\epsilon}\epsilon^{1-p}\right]^{\frac{D}{2p}}.
\end{equation}

\noindent Noticing that the RHS has its minimum for $\epsilon=\frac{p-1}{\log t}\in (0,1)$ (since $t> e$), we conclude that

$$\int_M |\overrightarrow{p_k}(t,x,y)| d\mu(y)\le C(t\log t)^{\frac{D(p-1)}{2p}},$$

\noindent which is the desired result.
\end{proof}


\noindent In the rest of this  section we assume (\ref{D3}), (\ref{G3}), $R_- \in \hat{K}$.  In addition, we suppose that $R_-$ is $\epsilon$-sub-critical, that is there exists $\epsilon\in[0,1)$ such that
\begin{equation}\tag{S-C}\label{SC3}
0\le (R_-\oo,\oo)\le\epsilon\,(H\oo,\oo), \forall \oo\in\mathcal{C}_0^{\infty}(\tm),
\end{equation}

\noindent where $H=\nabla^*\nabla+ R_+$.  Set $p_0:=\frac{2D}{(D-2)(1-\sqrt{1-\epsilon})}$ and fix $\tilde{p_0} < p_0$.

\begin{prop}\label{proposition41}
For all $t>0$ and $x,y\in M$
\begin{equation}\label{upperbound3}
|\overrightarrow{p_k}(t,x,y)|\le \frac{C(1+ t)^{\frac{D}{\tilde{p}_0}}}{v(x,\sqrt{t})}\exp\left(-c\frac{\rho^2(x,y)}{t}\right).
\end{equation}
\end{prop}

\begin{proof}
We proceed in three  steps. \\

\noindent \underline{Step 1}.  We show the $L^2-L^p$ estimates
\begin{equation}\label{equa38}
\underset{t>0}{\sup}\,\|e^{-t\overrightarrow{\Delta_k}}v(.,\sqrt{t})^{\frac{1}{2}-\frac{1}{p}}\|_{2-p}\le C
\end{equation}

\noindent for all $p\in[2,p_0)$. 

\noindent Let $j\in \mathbb{N}$ and let $A(x,\sqrt{t}, j)$ be  the annulus $B(x,(j+1)\sqrt{t})\setminus B(x,j\sqrt{t})$. Following the proof of \cite{M} Theorem 4.1 leads to the following $L^q-L^2$ off-diagonal estimates

$$\|\chi_{B(x,\sqrt{t})}e^{-t\overrightarrow{\Delta_k}}\chi_{A(x,\sqrt{t}, j)}\|_{q-2}\le \frac{C}{v(x,\sqrt{t})^{\frac{1}{q}-\frac{1}{2}}}e^{-cj^2}$$

\noindent for all $q\in(p_0',2]$. This implies that  for all $q\in(p_0',2]$

$$\underset{t>0}{\sup}\,\|v(.,\sqrt{t})^{\frac{1}{q}-\frac{1}{2}}e^{-t\overrightarrow{\Delta_k}}\|_{q-2}\le C.$$
See \cite{ao} Proposition 2.9. \\


\noindent \underline{Step 2}.  We prove that for all $t>0$
\begin{equation}\label{equa39}
\|v(.,\sqrt{t})^{\frac{1}{2}}e^{-t\overrightarrow{\Delta_k}}\|_{2-\infty}\le C (1+t)^{\frac{D}{2\tilde{p}_0}}.
\end{equation}

\noindent Let $0<t\le 1$. Using Theorem \ref{theorem1} and Lemma \ref{integral}, we  obtain easily 

$$\|v(.,\sqrt{t})^{\frac{1}{2}}e^{-t\overrightarrow{\Delta_k}}\|_{2-\infty}\le C\le C (1+t)^{\frac{D}{2\tilde{p}_0}}.$$

\noindent We now consider $t>1$. Since $\overrightarrow{\Delta_k}$ satisfies the $L^2-L^2$ Davies-Gaffney estimates (see e.g., \cite{S}, Theorem 6), a consequence  of \cite{bcs}, Proposition 4.1.6 is 

$$\|v(.,\sqrt{t})^{\frac{1}{2}}e^{-t\overrightarrow{\Delta_k}}\|_{2-\infty}\le C\|v(.,\sqrt{t})^{\frac{1}{\tilde{p}_0}}e^{-t\overrightarrow{\Delta_k}}v(.,\sqrt{t})^{\frac{1}{2}-\frac{1}{\tilde{p}_0}}\|_{2-\infty},$$ 

\noindent with $C$ independent of $t$. The semigroup property then gives

$$\|v(.,\sqrt{t})^{\frac{1}{2}}e^{-t\overrightarrow{\Delta_k}}\|_{2-\infty}\le C \|v(.,\sqrt{t})^{\frac{1}{\tilde{p}_0}}e^{-\frac{t}{2}\overrightarrow{\Delta_k}}\|_{\tilde{p}_0-\infty}\|e^{-\frac{t}{2}\overrightarrow{\Delta_k}}v(.,\sqrt{t})^{\frac{1}{2}-\frac{1}{\tilde{p}_0}}\|_{2-\tilde{p}_0}.$$

\noindent We use (\ref{equa38}) and (\ref{D3}) to obtain 

$$\|v(.,\sqrt{t})^{\frac{1}{2}}e^{-t\overrightarrow{\Delta_k}}\|_{2-\infty}\le C\|v(.,\sqrt{t/2})^{\frac{1}{\tilde{p}_0}}e^{-\frac{t}{2}\overrightarrow{\Delta_k}}\|_{\tilde{p}_0-\infty}.$$

\noindent Using again (\ref{D3}) and the semigroup property leads to

$$\|v(.,\sqrt{t})^{\frac{1}{2}}e^{-t\overrightarrow{\Delta_k}}\|_{2-\infty}\le C t^{\frac{D}{2\tilde{p}_0}}\|v(.,\sqrt{1/2})^{\frac{1}{\tilde{p}_0}}e^{-\frac{1}{2}\overrightarrow{\Delta_k}}\|_{\tilde{p}_0-\infty}\|e^{-\left(\frac{t}{2}-\frac{1}{2}\right)\overrightarrow{\Delta_k}}\|_{\tilde{p}_0-\tilde{p}_0}.$$

\noindent As in  \cite{M} Theorem 3.1, we prove that the semigroup $(e^{-t\overrightarrow{\Delta_k}})_{t\ge 0}$ is uniformly bounded on $L^{\tilde{p}_0}$ (the statement in \cite{M} is for $k=1$). Therefore 
\begin{align*}
\|v(.,\sqrt{t})^{\frac{1}{2}}e^{-t\overrightarrow{\Delta_k}}\|_{2-\infty}
&\le C t^{\frac{D}{2\tilde{p}_0}}\|v(.,\sqrt{1/2})^{\frac{1}{\tilde{p}_0}}e^{-\frac{1}{2}\overrightarrow{\Delta_k}}\|_{\tilde{p}_0-\infty}\\
&\le C(1+t)^{\frac{D}{2\tilde{p}_0}},
\end{align*}

\noindent where we use Theorem \ref{theorem1} and Lemma \ref{integral} to obtain the last inequality.  This concludes the proof of (\ref{equa39}).\\

\noindent \underline{Step 3}.  We finish the proof  by using  \cite{coulhonsikora}, Corollary 4.5 with 
$$V(x,t):= \frac{v(x,t)}{ (1+t)^{2D/\tilde{p_0}}}.$$
\end{proof}

Assertion $(iii)$ of Theorem \ref{theorem2} is exactly the statement of the previous proposition. Assertion $(iv)$ follows from the same proposition and Proposition \ref{semigaussian2}. The proof of assertion $(ii)$ for $p \ge p_0$ is similar to the proof of Proposition \ref{gradient}, one has to use   (\ref{upperbound3}) instead of (\ref{upperbound2}) and interpolation with 
\begin{equation}\label{interpolation4}
\|\nabla e^{-t\Delta}\|_{\tilde{p}_0-\tilde{p}_0}\le\frac{C}{\sqrt{t}}.
\end{equation}
This latter estimate is a consequence of the boundedness of the Riesz transform on $L^p$ for $p < p_0$ (cf. \cite{M}.
The proof of Theorem \ref{theorem2} is complete.   \\

Finally, we prove Proposition \ref{proposition1}.
\begin{proof}[Proof of Proposition \ref{proposition1}]
Under the assumption of the proposition it is proved in \cite{CMO}, Theorem 3.3 that the Riesz transform $d^* \ddd^{-1/2}$ is bounded from  $H^r_{\ddd}(\tm)$ to $L^r(M)$ for all $r \in (1, 2]$.  The  space $H^p_{\ddd}(\tm)$ is a Hardy space associated with $\ddd$. On the other hand, the semigroup $(e^{-t\ddd})_{t \ge0}$ satisfies off-diagonal estimates
\begin{equation}\label{offdiag}
\| \chi_{B(x, \sqrt{t})} e^{-t\ddd} \chi_{B(y,\sqrt{t})} \|_{r-2} \le C v(x,\sqrt{t})^{\frac{1}{2}- \frac{1}{r}} e^{-c \frac{\rho^2(x,y)}{t}}, \ t > 0, x, y \in M.
\end{equation}
This estimate is shown in \cite{M} for all $r \in (p_0', 2]$. The  limitation $r > p_0'$ follows from the fact that the semigroup $(e^{-t\ddd})_{t \ge0}$ is known to be  uniformly bounded on $L^r(\tm)$ only for $r \in (p_0', p_0)$. Now if the semigroup $(e^{-t\ddd})_{t \ge0}$ is uniformly bounded on $L^p(\tm)$ (and so by duality on $L^{p'}(\tm)$) then \eqref{offdiag} holds for all $r \in (\min(p,p'), 2]$. It is known that the validity of \eqref{offdiag} implies that $L^r(\tm)$ and  $H^r_{\ddd}(\tm)$ coincide and have equivalent norms (see Proposition 2.5 in \cite{CMO} and references there). This implies that $d^* \ddd^{-1/2}$  is bounded from $L^r(\tm)$
to $L^r(M)$ for all $r \in (\min(p,p'), 2]$. By duality and the commutation formula $d \Delta = \ddd d$ the Riesz transform $d \Delta^{-1/2}$ is bounded from 
$L^r(M)$ to $L^r(\tm)$ for all $r \in [2, \max(p,p'))$. The boundedness for $r \in (1, 2)$ is well known (cf. \cite{could}). 
\end{proof}

\noindent Institut de Math\'ematiques de Bordeaux, UMR 5251\\
Universit\'e de Bordeaux,\\
351,  Cours de la Lib\'eration 33405 Talence. France.\\
Jocelyn.Magniez@math.u-bordeaux.fr, \ Elmaati.Ouhabaz@math.u-bordeaux.fr


\begin{thebibliography}{99}

\bibitem{assaad}
J. Assaad,
Riesz transforms associated to Schr\"odinger operators with negative potentials,
\textit{Publ. Mat.},
55(1) (2011)  123-150. 



\bibitem{ao}
J. Assaad and E.M.  Ouhabaz,
Riesz transforms of Schr\"odinger operators on manifolds,
\textit{J. Geom. Anal.},
22(4) (2012)  1108-1136. 







\bibitem{acdh}
P. Auscher, Th. Coulhon, X.T. Duong and S. Hofmann,
Riesz transform on manifolds and heat kernel regularity,
\textit{Ann. Sci. Ecole Norm. Sup. (4)},
37(6) (2004)  911-957. 




\bibitem{B}
D. Bakry,
Etude des transformations de Riesz dans les vari\'et\'es riemanniennes \`a courbure de Ricci minor\'ee, In
\textit{S\'eminaire de Probabilit\'es, XXI}, volume 1247 of
\textit{Lecture Notes in Math.},
pages 137-172.
Springer, Berlin, 1987.


\bibitem{Berard}
P. B\'erard, From vanishing theorems to estimating theorems: the Bochner technique revisited, \textit{Bull. A.M.S.} 19 (1988), no 2, 371-406.
\bibitem{bcs}
S. Boutayeb, Th. Coulhon and A. Sikora, A new approach to pointwise heat kernel upper bounds on doubling metric measure spaces, \textit{Adv. Math.}, 270 : (2015) 302-374. 












\bibitem{C2}  G. Carron, Riesz  transform on manifolds with quadratic curvature decay,
2014, available at http://arxiv.org/abs/1403.6278. 



\bibitem{carron}
G. Carron, Th.  Coulhon and A.  Hassell,
Riesz transform and $L^p$-cohomology for manifolds with Euclidean ends,
\textit{Duke Math. J.},
133(1) (2006)  59-93. 






\bibitem{CMO} P. Chen, J. Magniez and E.M.  Ouhabaz, The Hodge-de Rham Laplacian and $L^p$-boundedness of Riesz transforms on non-compact manifolds, \textit{Nonlinear Analysis TMA}, 125 (2015) 78-98. 





\bibitem{CDS}
Th. Coulhon, B.  Devyver and A. Sikora, Gaussian heat kernel estimates: from functions to forms,  2016 available at http://arxiv.org/abs/1606.02423. 



\bibitem{could}
Th.  Coulhon and X.T.  Duong,
Riesz transforms for $1\leq p\leq 2$,
\textit{Trans. Amer. Math. Soc.},
351(3) (1999) 1151-1169. 





\bibitem{coulhonduong}
Th. Coulhon and X.T. Duong,
Riesz transform and related inequalities on noncompact Riemannian manifolds,
\textit{Comm. Pure Appl. Math.},
56(12) (2003) 1728-1751. 


\bibitem{coulhonsikora}
Th. Coulhon and A. Sikora,
Gaussian heat kernel bounds via the Phragm\`en-Lindel\"of theorem,
\textit{Proc. London Math. Soc.},
96 (2008) 507-544. 



\bibitem{coulz}
Th. Coulhon and Qi S. Zhang,
Large time behavior of heat kernels on forms,
\textit{J. Differential Geom.},
77(3)  (2007) 353-384. 





\bibitem{davies}
E. B. Davies,
\textit{Heat Kernels and Spectral Theory}, volume 92 of
\textit{Cambridge Tracts in Mathematics}.
Cambridge University Press,
Cambridge, 1989.




\bibitem{DP} E. B. Davies and M. M. H. Pang,  Sharp heat kernel bounds for some Laplace operators, {\it Quart. J.
Math. Oxford Ser.}  (2), 40, 159 (1989) 281–290. 




\bibitem{daviessimon}
E. B. Davies and B.  Simon, $L^p$ norms of non-critical Schr\"odinger semigroups, \textit{J. Funct. Anal.}, 102 (1991)  95-115.


\bibitem{davies97}
E.B. Davies, Non-Gaussian aspects of heat kernel behaviour,  \textit{J. London Math. Soc.} 55 (2) (1997) 105-125.




\bibitem{D}
B. Devyver,
A Gaussian estimate for the heat kernel on differential forms and application to the Riesz transform,
\textit{Math. Ann.},
358, no.1-2 (2004)  25-68. 




\bibitem{Gr2}  A. Grigor'yan, Upper bounds of derivatives of heat kernels on arbitrary complete manifolds, \textit{J. Funct. Anal.} 127 (1995) 363-389.


\bibitem{Gr1} A. Grigor’yan, Gaussian upper bounds for the heat kernel on arbitrary manifolds, {\it J.Differential Geom.}
  45(1) (1997) 33–52.





 



 
\bibitem{grigor}
A.  Grigor'yan,
\textit{Heat Kernel and Analysis on Manifolds}, volume 47 of
\textit{AMS/IP Studies in Advanced Mathematics}.
American Mathematical Society,
Providence, RI,
2009.


\bibitem{HSU}  H. Hess, R. Schrader and D.A. Uhlenbrock, Domination of semigroups and generalizations of Kato's inequality,   \textit{Duke M. J.} 44 (4) (1977) 
893-904. 




















	
	
	
	
	









\bibitem{M} J.  Magniez, Riesz transforms of the Hodge-de Rham Laplacian on Riemannian manifolds,  \textit{Math.  Nach.} 289 (2016), no. 8-9, 1021-1043.





\bibitem{ouhabaz}
E.M.  Ouhabaz,
\textit{Analysis of Heat Equations on Domains}, 
\textit{London Mathematical Society Monographs Series, Vol. 31}.
Princeton University Press,
Princeton, NJ, 2005.




\bibitem{Ouhabaz2} E.M.  Ouhabaz, 
Comportement des noyaux de la chaleur des opérateurs de Schr\"odinger et applications à certaines équations paraboliques semi-linéaires, 
\textit{J. Funct. Anal.},
238 (2006)  278-297. 


\bibitem{Ou98} E.M.  Ouhabaz, 
$L^p$-contraction semigroup for vector-valued functions, 
\textit{Positivity} 3 (1999) n. 1,  83-93. 












\bibitem{S}
A. Sikora,
Riesz transform, Gaussian bounds and the method of wave equation,
\textit{Math. Z.}, 247, no.3 (2004)  643-662. 






\bibitem{Simon} 
B. Simon, 
Schr\"odinger semigroups, 
\textit{Bull. Amer. Math. Soc.}, 7  (1982) 447-526. 












\bibitem{strichartz}
R. S. Strichartz,
Analysis of the Laplacian on the complete Riemannian manifold,
\textit{J. Funct. Anal.},
52, no.1 (1983)  48-79. 





\bibitem{T} M.  Takeda,
Gaussian bounds of heat kernels for Schr\"odinger operators on
Riemannian manifolds, {\it Bull. Lond. Math. Soc.} {\bf 39} (2007),
no. 1, 85--94.











\bibitem{voigt}
J. Voigt, Absorption semigroups, their generators, and Schr\"odinger semigroups, \textit{J. Funct. Anal.}, 67 (1986) 167-205. 




\end{thebibliography}
\end{document}